\documentclass[a4paper, english]{amsart}

\usepackage[utf8]{inputenc}

\usepackage[T1]{fontenc}

\usepackage{tikz}

\usepackage{eurosym}

\usepackage{graphicx}
\usepackage{shadethm}

\usepackage[all]{xy}
\SelectTips{cm}{12}

\usepackage{longtable}

\usepackage{amsmath}
\usepackage{amsfonts}
\usepackage{amssymb}
\usepackage{amsthm}
\usepackage{amstext}
\usepackage{float}

\usepackage{multicol}
\usepackage{color}

\usepackage{titlesec}
\titleformat{\section}[hang]{\Large\bfseries\scshape}{\thesection\quad}{0pt}{}
\titleformat{\subsection}[hang]{\large\bfseries\scshape}{\thesubsection\quad}{0pt}{}
\titleformat{\subsubsection}[hang]{\bfseries\scshape}{\thesubsubsection\quad}{0pt}{}
\titleformat{\paragraph}[hang]{\bfseries\scshape}{\theparagraph\quad}{0pt}{}

\usepackage{setspace}

\usepackage{algorithm}
\usepackage{algorithmic}

\floatname{algorithm}{\textsc{Algorithm}}

\newtheorem{theorem}{\textsc{Theorem}}[section]
\newtheorem{lemma}[theorem]{\textsc{Lemma}}
\newtheorem{proposition}[theorem]{\textsc{Proposition}}

\newenvironment{definition}[1][\textsc{Definition.}]
  {\begin{trivlist}
    \item[\hskip \labelsep {\bfseries #1}]}
  {\end{trivlist}}

\newenvironment{remark}[1][\textsc{Remark.}]
  {\begin{trivlist}
    \item[\hskip \labelsep {\bfseries #1}]}
  {\end{trivlist}}



\newcommand{\N}{\mathbb N}
\newcommand{\Z}{\mathbb Z}
\newcommand{\Q}{\mathbb Q}

\newcommand{\C}{\mathbb C}

\newcommand{\F}{\mathbb F}

\newcommand{\Fpb}{\bar{\mathbb F}_p}

\renewcommand{\O}{\mathcal O}

\newcommand{\calA}{\mathcal A}

\newcommand{\frakP}{\mathfrak{P}}

\newcommand{\Cl}{\mathcal {C} \ell}
\newcommand{\Ell}{\mathcal E \ell \ell}

\newcommand{\End}{\operatorname{End}}
\newcommand{\Aut}{\operatorname{Aut}}

\newcommand{\Gal}{\operatorname{Gal}}

\newcommand{\Norm}{\operatorname{Norm}}

\begin{document}

  \title{Computing Isogenies between Supersingular Elliptic Curves over $\F_p$}

  \author{Christina Delfs}
  \address{Carl von Ossietzky Universit{\"a}t Oldenburg}
  \email{christina.delfs@uni-oldenburg.de}

  \author{Steven D. Galbraith}
  \address{University of Auckland}
  \email{s.galbraith@math.auckland.ac.nz}

  \date{October 25, 2013}

  \thispagestyle{empty}

  \begin{abstract}
    Let $p>3$ be a prime and let $E$, $E'$ be supersingular elliptic curves over $\F_p$. We want to construct an isogeny $\phi:E\to E'$. The currently fastest algorithm for finding isogenies between supersingular elliptic curves solves this problem in the full supersingular isogeny graph over $\F_{p^2}$.
    It takes an expected $\tilde\O(p^{1/2})$ bit operations, and also $\tilde\O(p^{1/2})$ space, by performing a ``meet-in-the-middle'' breadth-first search in the isogeny graph.

    In this paper we consider the structure of the isogeny graph of supersingular elliptic curves over $\F_p$. We give an algorithm to construct isogenies between supersingular curves over $\F_p$ that works in $\tilde\O(p^{1/4})$ bit operations. We then discuss how this algorithm can be used to obtain an improved algorithm for the general supersingular isogeny problem.
  \end{abstract}

  \maketitle


\section{Introduction}
\label{section_int}

The problem of computing an isogeny between two given elliptic curves has been studied by many authors and has several applications \cite{kohel1996endomorphism, galbraith1999constructing, galbraith2002extending, jao2005all, charles2009cryptographic, stolbunov2010constructing, jao2009expander, jao2011towards}. A natural question that has not previously been considered is to construct isogenies between two given supersingular elliptic curves over $\F_p$.

Let $p$ be a prime, $q := p^n$ for some integer $n$ and let $L$ be a non-empty set of small primes with $p\notin L$.
Let $K$ be either $\F_q$ or $\Fpb$.
The \emph{supersingular isogeny graph $X( K, L)$} is a directed graph where the vertices are $K$-isomorphism classes of supersingular elliptic curves over $\F_q$ and the edges are equivalence classes of $\ell$-isogenies defined over $K$ between such curves for $\ell\in L$. (Two isogenies are equivalent if they have the same kernel.) If we only consider $L=\{\ell\}$, we write $X(K, \ell)$. Usually the vertices are represented by $j$-invariants.

For various reasons we always assume $p > 3$ in this paper.  Note that $X(\Fpb, L )$ has only one vertex when $p < 11$ and so all the problems we consider are trivial for small $p$.

If we regard the full supersingular isogeny graph $X(\Fpb,\ell)$, it suffices to consider elliptic curves defined over $\F_{p^2}$, since the $j$-invariant of a supersingular elliptic curve always lies in $\F_{p^2}$. In general the isogenies will still be defined over $\bar\F_p$ though.

Let $S_{p^2}$ be the set of all supersingular $j$-invariants in $\F_{p^2}$. It is well-known (e.g. \cite[Theorem V.4.1(c)]{silverman2009arithmetic}) that for a prime $p>3$ we have
\begin{eqnarray*}
  \# S_{p^2} = \left\lfloor \frac{p}{12} \right\rfloor +
  \begin{cases}
    0 & \text{if } \ p \equiv 1 \pmod{12}, \\
    1 & \text{if } \ p \equiv 5, 7 \pmod{12}, \\
    2 & \text{if } \ p \equiv 11 \pmod{12}.
  \end{cases}
\end{eqnarray*}

In contrast to the ordinary case, the graph $X(\bar\F_p,\ell)$ has an irregular structure but is always fully connected for every prime $\ell$ (see \cite{mestre1986methode} or \cite[Corollary 78]{kohel1996endomorphism}). Thus we can use a chain of isogenies of small prime degree (i.e. $\ell=2$) to construct an isogeny between two given supersingular elliptic curves over $\F_{p^2}$. Those isogenies are fast to compute.
These graphs are known to be expanders (see~\cite{charles2009cryptographic} for references), so they have small diameter and there is a short path between any two vertices.
A natural problem is to find a path between any two vertices in the graph.

A general ``meet-in-the-middle'' idea for finding paths in graphs (also called ``bi-directional search''), was proposed by Pohl~\cite{Poh69}. (Indeed, this finds shortest paths.)
This idea was used by Galbraith~\cite{galbraith1999constructing} to construct isogenies between elliptic curves, and it is applicable for both ordinary and supersingular curves. The problem is that the algorithm requires large storage, and is not easy to parallelise.
In the ordinary case, a low-storage and parallelisable algorithm was proposed by Galbraith, Hess and Smart~\cite{galbraith2002extending}  and improved by Galbraith and Stolbunov~\cite{galbraith2011improved}. (These algorithms are no longer guaranteed to find the shortest path.)
While some of these ideas can be adapted to get a low-storage parallel algorithm for the supersingular isogeny problem, there are several reasons why the supersingular case is more awkward than the ordinary case: we might wish to use just $L = \{ 2 \}$ and then it is hard to prevent short cycles and walks are ``not random enough'';  unlike~\cite{galbraith2002extending} there is no process to ``shorten'' or ``smooth'' a long walk.
Hence, it has remained an open problem to get a good isogeny algorithm for the supersingular graph $X( \Fpb,  2  )$.

The subgraph of $X( \Fpb,  2  )$ we get through deleting the vertices where the $j$-invariants are in $\F_{p^2}\setminus\F_p$ is considerably smaller. For a prime $p>3$ let $S_p$ be the set of all supersingular $j$-invariants in $\F_p$. Then
\begin{eqnarray} \label{eq:Sp}
  \# S_p = \begin{cases}
    \tfrac{1}{2}h(-4p) & \text{if } \ p \equiv 1 \pmod 4 \\
    h(-p) & \text{if } \ p \equiv 7 \pmod 8 \\
    2 h(-p) & \text{if } \ p \equiv 3 \pmod 8
  \end{cases}
\end{eqnarray}
where $h(d)$ is the class number of the imaginary quadratic field $\Q(\sqrt{d})$. This may be proved using a counting argument (see \cite[Theorem 14.18]{Cox}).

Since the class number of an imaginary quadratic field $K$ with discriminant $d_K$ can be bounded as $h_K \leq \tfrac{1}{\pi} \sqrt{|d_K|}\ln{|d_K|}$ (see \cite[Exercise 5.27]{cohen1996course}), the size of this set is $\tilde\O(\sqrt{p})$, so we expect shorter paths when working in this smaller graph and thus faster algorithms for constructing isogenies. The problem is that in general those graphs are not connected, and hence it is not always possible to obtain an isogeny with degree a power of $\ell$ between arbitrary supersingular elliptic curves over $\F_p$ without going via elliptic curves over $\F_{p^2}$.

So there are two questions arising about this subgraph:
\begin{itemize}
  \item How many prime isogeny degrees $\ell\in L$ do we have to allow until the subgraph of supersingular elliptic curves over $\F_p$ is connected?
  \item Is there an algorithm for computing an isogeny between supersingular elliptic curves over $\F_p$ which is faster than the known algorithms for the full graph $X( \Fpb, L )$?
\end{itemize}

We will answer both questions in the course of this work.
Section~\ref{section_graphs} explains the structure of the graph $X( \F_p, L )$.
The main observation is that the supersingular case restricted to $\F_p$ closely resembles the ordinary case, and so the known advantages of that situation can be exploited for supersingular curves too.
Section~\ref{section_iso_prob} presents an algorithm, arising from these considerations, that computes isogenies between supersingular elliptic curves over $\F_p$.
Section~\ref{section_on} explains how our methods also lead to a good solution to the general isogeny problem (i.e., in the full graph $X( \Fpb, L )$).
A few toy example graphs over $\bar\F_p$ and $\F_p$ are given in the appendix to illustrate the results of Section~\ref{section_graphs}.

\section{The Structure of Supersingular Isogeny Graphs}
\label{section_graphs}

We first make some remarks about the relation between $X( \F_p, L )$ and $X( \Fpb, L )$. Importantly, with our definitions, the former is not a subgraph of the latter.

An ordinary elliptic curve over $\F_p$ is never isogenous to its non-trivial quadratic twist since they have a different number of $\F_p$-rational points, so we never have to care about twists when considering isogenies between ordinary elliptic curves over $\F_p$. If the curves are supersingular though, this is not the case.

Let $p>3$. A supersingular elliptic curve over $\F_p$ has $p+1$ points and so all quadratic twists have the same number of points. Thus the twists are isogenous but lie in different $\F_p$-isomorphism classes. Therefore it is not very precise to represent the vertices in the supersingular isogeny graph over $\F_p$ with $j$-invariants, since then the different isomorphism classes collapse to only one vertex and the picture of in- and outgoing isogenies is distorted. So if we want to differentiate between twists, we have to store more information than just the $j$-invariants of the elliptic curves, for instance the quantities $c_4(E), c_6(E)$ in addition to $j(E)$. We will see that we have twice the number of vertices in $X(\F_p, L )$ as in $X(\bar\F_p, L) \cap \F_p$.

In this situation it is no longer possible to compute the neighbours of a given vertex using only the modular polynomial, since this only produces the $j$-invariant of the image curve and does not keep track of twists. Instead we use the formulae of V{\'e}lu \cite{velu1971isogenies} to compute the image curve under an isogeny whose kernel is a Galois-invariant subgroup $G$ of $E$ with order $\ell$. These subgroups can be constructed from factors of the $\ell^\text{th}$ division polynomial or using a basis of $E[\ell]$. This method was used to construct the graphs in the appendix.

If we construct the supersingular $\ell$-isogeny graph over $\F_p$ regarding these considerations, a much more regular structure appears as demonstrated in the Appendix for some examples with small $p$ and $\ell$. On closer examination the graphs $X(\F_p, \ell)$ resemble the ``volcano''-structure of the ordinary case, only that here we have mostly ``craters'', that is, isogeny-circles. We want to describe now why this structure appears.

As in the ordinary case, the properties of isogenies between supersingular elliptic curves over $\F_p$ are strongly connected to the structure of their endomorphism rings. We define $\End_{\F_q} E$ to be the ring of endomorphisms of $E$ that are defined over $\F_q$. In general we know that the endomorphism ring of an elliptic curve over $\F_q$ is an order in a division algebra $\calA := \End_{\F_q} E \otimes_\Z \Q$. Depending on the number of $\F_q$-rational points there are some more precise results about this as can be seen in the next theorem (see \cite{waterhouse1969abelian} or \cite{rueck1987note}).

\begin{theorem}
  \label{thm_end}
  Let $p>3$, $q = p^n$ and $E$ be a supersingular elliptic curve over $\F_{q}$ with $\#E(\F_{q}) = q+1-t$ where and $|t|\leq 2\sqrt{q}$. Then one of the following cases must be true:
  \begin{enumerate}
    \item $n$ is even and $t = \pm2\sqrt q$,
    \item $n$ is even, $p\not\equiv 1 \pmod 3$ and $t=\pm\sqrt q$,
    \item $n$ is even and $p\not\equiv 1 \pmod 4$ and $t=0$,
    \item $n$ is odd and $t=0$,
  \end{enumerate}

  In this situation the corresponding division algebra $\calA$ is also determined by the cases. Let $\pi_q$ be the $q$-th power Frobenius endomorphism.

  In the first case $\calA$ is a quaternion algebra over $\Q$, $\pi_q$ is a rational integer and $\End_{\F_q} E$ is a maximal order in $\calA$.

  In the other three cases $\calA = \Q(\pi_q)$ is an imaginary quadratic field over $\Q$ and $\End_{\F_q} E$ is an order in $\calA$ with conductor prime to $p$.
\end{theorem}

Now, if we take a supersingular elliptic curve $E$ over $\F_p$ with $p>3$, we end up in Case~4 of Theorem~\ref{thm_end}. Thus we know according to the theorem that $\End_{\F_p} E$ is an order in $K := \Q(\pi_p)$ and its conductor is prime to $p$. Since $\pi_p^2+p = 0$ holds, we get $K = \Q(\sqrt{-p})$. Furthermore
\begin{eqnarray*}
  \Z\left[\pi_p\right] = \Z\left[\sqrt{-p}\right] = \Z\left[\tfrac{d+\sqrt{d}}{2}\right] \subseteq \End_{\F_p} E \subseteq \Z\left[\tfrac{d_K+\sqrt{d_K}}{2}\right] = \O_K
\end{eqnarray*}
has to hold where $d=-4p$, $\O_K$ is the maximal order and $d_K$ the fundamental discriminant of $K$. Due to the properties of the fundamental discriminant, we have $d = c^2\cdot d_K$ where $c\in\N$ is maximal such that $d_K \equiv 0,1 \pmod 4$ and is called the \emph{conductor of $\Z[\pi_p]$ in $\O_K$}. From these observations we can conclude:

\begin{itemize}
  \item If $p \equiv 1 \pmod 4$, we always get $d_K = d = -4p$, $\Z[\pi_p] = \O_K$ and hence $\End_{\F_p} E = \Z\left[\sqrt{-p}\right]$ for a supersingular elliptic curve $E$ over $\F_p$.
  \vspace{1ex}
  \item If $p \equiv 3 \pmod 4$, we get $d_K = -p$. Thus $\Z[\pi_p] = \Z\left[\sqrt{-p}\right]$ has conductor $c=2$ in $\O_K = \Z\left[\tfrac{1+\sqrt{-p}}{2}\right]$ and $\End_{\F_p} E$ must be one of those two orders.
\end{itemize}

In terms of isogeny-volcanoes we can say that we have at most two levels. We will use the following terminology.

\begin{definition}
  Let $E$ be a supersingular elliptic curve over $\F_p$. We say \emph{$E$ is on the surface} (resp.  \emph{$E$ is on the floor}) if $\End_{\F_p} E = \O_K$ (resp.  $\End_{\F_p} E = \Z[\sqrt{-p}]$). Note that for $p\equiv 1\pmod 4$ surface and floor coincide.

  Let $\phi$ be an $\ell$-isogeny between supersingular elliptic curves $E$ and $E'$ over $\F_p$. If $\End_{\F_p} E \cong \End_{\F_p} E'$, then $\phi$ is called \emph{horizontal}. If $E$ is on the floor and $E'$ is on the surface (resp. $E$ on the surface and $E'$ on the floor), $\phi$ is called  \emph{$\ell$-isogeny up} (resp. \emph{down}).
\end{definition}

In the supersingular case there are fewer possibilities for $\ell$-isogenies up and down than for ordinary volcanoes (though, even in the ordinary case tall volcanoes are quite rare). This is due to the fact that for an isogeny $\phi:E\to E'$ with $[\End E:\End E'] = \ell$ we have $\ell\mid\deg\phi$ (see \cite[Propositions 21 and 22]{kohel1996endomorphism}). Since in our case $[\End E:\End E'] \in\{1, 2\}$ (resp. $[\End E':\End E] \in\{1,2\}$) we get the following statement.

\begin{lemma}
  Let $\phi$ be a non-horizontal isogeny between supersingular elliptic curves over $\F_p$. Then the degree of $\phi$ is divisible by $2$.
\end{lemma}

Therefore we have no isogenies of odd prime degree going up or down in this graph.

To determine how many isogenies there are we need some theory about the ideal class group. We recall the relevant background below.

First we can make an observation about the number of $\F_p$-isomorphism classes of supersingular elliptic curves over $\F_p$ with a given $j$-invariant, based on the following proposition which follows directly from \cite[Theorem 2.2]{broker2006constructing}.

\begin{proposition}
  \label{prop_number}
  Let $p>3$ be a prime and $j \in \F_p$. Define $C_{p, j}$ as the set of $\F_p$-isomorphism classes of elliptic curves defined over $\F_p$ with $j$-invariant $j$. Then we get
  \begin{eqnarray*}
    \#C_{p, j}
    = \begin{cases}
        6 & j=0 \text{ and } p \equiv 1 \pmod 3 \\
        4 & j=1728 \text{ and } p \equiv 1 \pmod 4 \\
        2 & \text{otherwise.}
    \end{cases}
  \end{eqnarray*}
\end{proposition}

Since we know that for an elliptic curve $E$ over $\F_p$ with
\begin{eqnarray*}
  j(E) = 0: \hspace{1em} E \text{ is supersingular}& \iff & p \equiv 2\pmod 3 \\
  j(E) = 1728:  \hspace{1em} E \text{ is supersingular} & \iff & p \equiv 3\pmod 4
\end{eqnarray*}
holds, we can deduce from Proposition~\ref{prop_number} that given a supersingular $j$-invariant $j$ there are always exactly two $\F_p$-isomorphism classes of elliptic curves over $\F_p$ with this $j$-invariant.

\begin{proposition} \label{prop:j-in-Fp}
  Let $p > 3$ and let $E$ be a supersingular elliptic curve over $\Fpb$. Then
  \begin{eqnarray*}
    E \text{ is defined over } \F_p \iff \Z[\sqrt{-p}]\subseteq \End E.
  \end{eqnarray*}
\end{proposition}

\begin{proof}
  The implication $(\Rightarrow)$ is immediate since $\pi_p$ lies in $\End E$. To prove the implication $(\Leftarrow)$, let $\psi \in \End E$ satisfy $\psi^2 = [-p]$. Then $\psi$ is an isogeny of degree $p$ and $\widehat{\psi} \circ \psi = [p]$. Since $E$ is supersingular it follows that $\psi$ has kernel $\{ \O_E \}$ and so is inseparable.  Hence, by Corollary~II.2.12 of~\cite{silverman2009arithmetic} , $\phi$ composes as
\[
   E \stackrel{\pi}{\longrightarrow} E^{(p)} \stackrel{\lambda}{\longrightarrow} E
\]
where $\pi$ is the $p$-power Frobenius map and $E^{(p)}$ is the image curve of Frobenius.  Now $\deg( \lambda ) = 1$ and so $\lambda$ is an isomorphism. Hence, $j(E) = j( E^{(p)} ) = j(E)^p$.  Hence, $j(E) \in \F_p$.
\end{proof}

Now we want to give a connection between supersingular elliptic curves over $\F_p$ and certain elliptic curves in characteristic $0$. In the ordinary case, the Deuring Reduction Theorem gives a one-to-one correspondence preserving the endomorphism ring. In the supersingular case, since $\End_{\Fpb} E $ is too large, it is less clear how to construct such a correspondence.
But, for the case where $E$ is defined over $\F_p$ with $p > 3$, then we have seen that $\End_{\F_p} E $ is an order in the imaginary quadratic field $\Q( \sqrt{-p} )$. Hence, we can hope to get an analogous one-to-one correspondence.  We now show this is the case.

\begin{proposition}\label{lem:correspondence}
  There is a one-to-one correspondence
  \begin{eqnarray*}
    \left\{\begin{matrix}
      \text{supersingular elliptic} \\
      \text{curves over }\F_p
    \end{matrix}\right\}
    & \longleftrightarrow &
    \left\{\begin{matrix}
      \text{elliptic curves }E\text{ over } \C \\
      \text{with } \End E \in \{\Z[\sqrt{-p}], \O_K\}
    \end{matrix}\right\}.
  \end{eqnarray*}
\end{proposition}

\begin{proof}
  Let $\Ell_p(\C)$ be the set of isomorphism classes of supersingular elliptic curves in characteristic $0$ with endomorphism ring $\O \in \left\{\Z[\sqrt{-p}], \O_K\right\}$. Every element in $\Ell_p(\C)$ corresponds to an ideal class in $\Cl(\O)$, so due to the observations above we get
  \begin{eqnarray*}
    \#\Ell_p(\C) & = & \sum _{\text{possible }\O}\#\Cl(\O) \\
    & = & \begin{cases}
            h(-4p) & \text{if } p \equiv 1 \pmod 4 \\
            h(-4p) + h(-p) & \text{if } p \equiv 3 \pmod 4
          \end{cases} \\
    & = & \begin{cases}
            h(-4p) & \text{if } p \equiv 1 \pmod 4 \\
            2h(-p) & \text{if } p \equiv 7 \pmod 8 \\
            4h(-p) & \text{if } p \equiv 3 \pmod 8
          \end{cases} \\
    & = & 2\#S_p.
  \end{eqnarray*}

  We also define $\Ell(\F_p)$ as the set of supersingular elliptic curves over $\F_p$ up to $\F_p$-isomorphism. We want to show that there is a bijective map
  \begin{eqnarray*}
    \Ell_p(\C) & \to & \Ell(\F_p) \\
    \left[E\right] & \mapsto & \left[\bar E\right]
  \end{eqnarray*}
  where $\bar E$ is the reduction of $E$ at some fixed place $\frakP_0$ over $p$.

  \begin{itemize}
    \item \textbf{Surjectivity:} \\
      Take a supersingular elliptic curve $\bar E$ over $\F_p$. Since Frobenius satisfies the polynomial $\pi_p^2 + p = 0$ it follows that $\End_{\F_p} \bar E$ contains the ring $\Z[\sqrt{-p}]$ or $\O_K$.
      Write this ring as $\Z[ \bar \psi ]$, so that $\bar \psi$ is either $\pi_p$ or $ (1 + \pi_p)/2$.
      Deuring's Lifting Theorem states that one can lift the pair $(\bar E, \bar \psi )$ to a pair $(E, \psi )$ where $E$ is an elliptic curve over some number field $H$ and $\psi \in \End( E )$ satisfies the same characteristic polynomial as $\bar \psi$. Indeed, $H$ is the Hilbert class field of $K$ or the ring class field of $\Z[\sqrt{-p}]$. Further, there is a place $\frakP$ of $H$ over $p$ such that the reduction of $E$ modulo $\frakP$ is isomorphic to $\bar E$.

      We want to show that reduction modulo a fixed place $\frakP_0$ of $H$ is surjective.  By Proposition~1.2 of~\cite{Tate67} there exists $\sigma \in \Gal( H/ K)$ such that $ \frakP^\sigma  = \frakP_0$ and so $E^\sigma$ reduces modulo $\frakP_0$ to the original curve $\bar E$.
      Hence, reduction modulo $\frakP_0$ is surjective.

    \item \textbf{Injectivity:} \\
      We see from equation~(\ref{eq:Sp}) and Proposition~\ref{prop_number} that $\#\Ell(\F_p) = 2\#S_p = \#\Ell_p(\C)$. Injectivity thus follows from surjectivity.
  \end{itemize}
\end{proof}

We re-inforce the fact that the correspondence of Proposition~\ref{lem:correspondence} is given by the Deuring lifting theorem: Given a supersingular elliptic curve $\bar E$ over $\F_p$ one performs Deuring lifting of the pair $(E, \psi )$ where $\psi = \pi_p$ or $\psi = (1 + \pi_p)/2 $.

It is important to see that isogenies behave well under this reduction. From Proposition 4.4 of \cite{silverman1994advanced} we know that reduction of isogenies is injective and preserves degrees. Furthermore we can show the following.

\begin{proposition}  \label{prop_rational}
  Let $\bar E_1, \bar E_2$ be supersingular elliptic curves over $\F_p$ and let $(E_1, \psi_1)$ and $( E_2, \psi_2 )$ be the Deuring lifts of $(\bar E_1, \pi_p)$ and $(\bar E_2, \pi_p)$ to characteristic $0$.
  Suppose there is an isogeny $\phi:E_1\to E_2$. Then the reduced isogeny $\bar\phi :\bar E_1\to\bar E_2$ is defined over $\F_p$.
\end{proposition}

\begin{proof}
  We may choose embeddings $\End(E_j) \to \C$ so that $\psi_j $ is identified with $i\sqrt{p}$.
  Now every isogeny $\phi: E_1 \to E_2$ will satisfy $\phi \circ i \sqrt p = i \sqrt p \circ \phi$ since in characteristic $0$ the isogenies correspond to multiplication with a complex number. After reduction the isogeny $\bar\phi :E_1 \to E_2$ commutes with $\pi_p$. Since the Frobenius generates the Galois group we get ${\bar\phi}^\sigma = \bar\phi$ for all $\sigma\in\Gal(\bar\F_p/\F_p)$ and thus $\bar\phi$ is defined over $\F_p$.
\end{proof}

To describe the structure of $X(\F_p, \ell)$, we begin with a supersingular elliptic curve $\bar E$ over $\F_p$. As we have seen it has $\O := \End_{\F_p} \bar E \in\{\Z[\sqrt{-p}], \O_K\}$ in $K := \Q(\sqrt{-p})$, which is an order of discriminant $d_E\in\{-4p, -p\}$. Via the Deuring Lifting Theorem this elliptic curve can be lifted to an elliptic curve $E$ over some number field with $\End E = \O$.

A standard fact (see~\cite{kohel1996endomorphism} or Theorem~4 of~\cite{galbraith1999constructing}) is the following.
Let $K = \Q( \sqrt{-p} )$ have discriminant $d_K$ and let $E$ be an elliptic curve over $\C$ with $\End( E) = \O$ being an order in $K$.
Let $c = [ \O_K : \O ]$ be the conductor of $\O$, so that the discriminant of $\O$ is $c^2 d_K$.
Let $\ell $ be a prime. Then there are the possibilities
\begin{enumerate}
  \item
    $\ell\mid c$: one isogeny up and $\ell$ isogenies down,
  \item
    \begin{enumerate}
      \item $\ell$ \textbf{splits in} $K$:  two horizontal isogenies and $\ell-1$ isogenies down,
      \item $\ell$ \textbf{is ramified in} $K$:  one horizontal isogeny and $\ell$ isogenies down,
      \item $\ell$ \textbf{is inert in} $K$:  $\ell+1$ isogenies down.
    \end{enumerate}
\end{enumerate}

The structure of isogenies up and down is called a volcano.
For our case, the only possibilities are  $c =1$ and, if $p \equiv 3 \pmod{4}$, $c=2$.
Hence, the only prime of interest is $\ell = 2$, and when $p \equiv 3 \pmod{4}$ and $\End(E) = \Z[ \sqrt{-p} ]$ then we have one $2$-isogeny up and two $2$-isogenies down.

Using these results we can construct an infinite ``volcano'' of elliptic curves whose endomorphism rings are orders in $\Q( \sqrt{-p} )$.
One can then consider the reduction modulo $p$ of this volcano.   All the reduced curves are supersingular.
What happens is that most of the curves do not reduce to elliptic curves defined over $\F_p$, hence only a finite part of the volcano survives in the graph $X(\F_p, \ell)$.
Proposition~\ref{prop:j-in-Fp} explains why some isogenies down do not appear in the supersingular isogeny graph over $\F_p$. Some of the isogenies down are to orders like $\Z[\ell\sqrt{-p}]$ which do not contain $\sqrt{-p}$ anymore. So those reduced curves are not defined over $\F_p$ and do not show up as vertices in the graph $X(\F_p, \ell)$.

As shown in Proposition~\ref{prop_rational}, the isogenies reduce to $\F_p$-rational outgoing isogenies between supersingular elliptic curves over $\F_p$. It remains to be shown that every such isogeny can be reached in that way.

Let $E$ be a supersingular elliptic curve over $\F_p$, so the $\F_p$-rational $\ell$-isogenies correspond to Galois-invariant cyclic subgroups of $E[\ell]$. When we consider some prime $\ell\neq p$, $E[\ell]$ is a $2$-dimensional vector space over $\F_\ell$ and $\pi_p$ acts linearly on $E[\ell]$. Fixing a basis $\{P, Q\}$ for $E[\ell]$, the action of $\pi_p$ is represented by a $2\times2$ matrix. We also know that $\pi_p^2 + p = 0$, so the matrix satisfies that characteristic polynomial modulo $\ell$.

There are three cases for the quadratic polynomial modulo $\ell$:
\begin{itemize}
  \item it factors as $(\pi_p - a)^2$,
  \item it factors as $(\pi_p - a)(\pi_p - b)$ with $a \not\equiv b \pmod \ell$,
  \item it is irreducible.
\end{itemize}

Suppose there is a cyclic subgroup $G = \langle P\rangle$ of $E[\ell]$ with $\pi_p(G) = G$. Then it follows that $\pi_p(P) = [a]P$ in $G$ for some integer $a$. We have that the linear map $\pi_p$ has eigenspace $\langle P\rangle$ with eigenvalue $a$ and so the characteristic polynomial has a root $(\pi_p - a)$.  Hence, changing basis to use the point $P$ and some other point $Q$ it then follows by standard linear algebra that $\pi_p$ is represented either by the matrix $\begin{pmatrix} a & 0 \\ b & a \end{pmatrix}$ or $\begin{pmatrix} a & 0 \\ 0 & b \end{pmatrix}$.

We then deduce that the number of Galois-invariant cyclic subgroups of $E[\ell]$ is in the first case $1$ or $\ell+1$ depending on whether the matrix's lower left entry is $b \ne 0$ or $b = 0$. In the second case we have two of them and in the third there are none.

The polynomial $x^2 + p \pmod \ell$ can only have a repeated root for $\ell = 2$. When $b$ in the matrix equals $0$ resp. $1$ modulo $2$, we get three or one Galois-invariant subgroups of $E[2]$. We want to show that those possibilities occur in the right cases, so when $\End_{\F_p} E = \Z[\tfrac{-p+\pi_p}{2}]$ we have three outgoing $2$-isogenies and when $\End_{\F_p} E = \Z[\pi_p]$ there is one of them.

We have $b\equiv 0\pmod 2$ if and only if $\pi_p(P) = P$ and $\pi_p(Q) = Q$, so $E[2] = \ker([2])$ is included in $\ker( 1+\pi_p )$. Since the multiplication-by-$2$-map is separable, there exists an unique isogeny $\phi\in\End E$ such that $1+\pi_p = 2\phi$ due to \cite[Corollary III.4.11]{silverman2009arithmetic}. $\phi$ is $\F_p$-rational since it is a quotient of $\F_p$-rational maps and therefore $\phi\in\End_{\F_p} E$. So the above is equivalent to $\Z[\pi_p] \subsetneq \End_{\F_p} E$ as we wanted to show.

For any other $\ell$ we get no $\F_p$-rational isogenies when the polynomial is irreducible, or two cyclic Galois-invariant subgroups when it is split. Finally, we are only interested in $\ell$ such that $\left(\tfrac{-p}{\ell}\right) = 1$, since otherwise there are no prime ideals of norm $\ell$ in $\Z[\sqrt{-p}]$ and so there are no edges in that graph. In that case we can see that the polynomial always splits with two distinct roots.

If we compare these results with the structure of the graph in characteristic $0$, we have exactly the same number of outgoing $\ell$-isogenies from the elliptic curves in both graphs which correspond to each other under reduction. Since the isogenies in characteristic $0$ reduce to $\F_p$-rational isogenies, there is a correspondence

\vspace{-2ex}
\begin{eqnarray*}
  \left\{\begin{matrix}
    \F_p \text{-rational }\ell\text{-isogenies} \\
    \text{between supersingular} \\
    \text{elliptic curves over }\F_p
  \end{matrix}\right\}
  & \longleftrightarrow &
  \left\{\begin{matrix}
    \ell\text{-isogenies between} \\
    \text{elliptic curves }E\text{ over } \C \\
    \text{with } \End E \in \{\Z[\sqrt{-p}], \O_K\}
  \end{matrix}\right\}.
\end{eqnarray*}

Thus we can transfer the picture from characteristic $0$ exactly to our graphs $X(\F_p, \ell)$ and with these considerations we have described their structure completely. It can be summed up in the following way:

\begin{theorem}
  \label{thm_struct}
  Let $p>3$ be a prime.
  \begin{enumerate}
    \item
      $p \equiv 1 \pmod 4$: There are $h(-4p)$ $\F_p$-isomorphism classes of supersingular elliptic curves over $\F_p$, all having the same endomorphism ring $\Z[\sqrt{-p}]$. From every one there is one outgoing $\F_p$-rational horizontal $2$-isogeny as well as two horizontal $\ell$-isogenies for every prime $\ell>2$ with $\left(\tfrac{-p}{\ell}\right) = 1$.
    \item
      $p \equiv 3 \pmod 4$: There are two levels in the supersingular isogeny graph. From each vertex there are two horizontal $\ell$-isogenies for every prime $\ell>2$ with $\left(\tfrac{-p}{\ell}\right) = 1$.
      \begin{enumerate}
        \item
          If $p\equiv 7 \pmod 8$, on each level $h(-p)$ vertices are situated. Surface and floor are connected 1:1 with $2$-isogenies and on the surface we also have two horizontal $2$-isogenies from each vertex.
        \item
          If $p\equiv 3 \pmod 8$, we have $h(-p)$ vertices on the surface and $3h(-p)$ on the floor. Surface and floor are connected 1:3 with $2$-isogenies, and there are no horizontal $2$-isogenies.
      \end{enumerate}
  \end{enumerate}
\end{theorem}

This provides a structure analogous to the one for the ordinary isogeny volcano, only that in our case we have no more than two levels and for $\ell>2$ only exactly two outgoing isogenies from each elliptic curve (if any). If $p\equiv 1\pmod 4$, there is only one $2$-isogeny starting from every vertex, whereas for $p\equiv 3\pmod 4$ we get two more when the elliptic curve is on the surface. Examples of all three cases are given in the Appendix. This result can be used to adapt the algorithms from the ordinary case that rely on the volcano structure. We will investigate one of them briefly in the next section.

\section{The Supersingular Isogeny Problem}
\label{section_iso_prob}

We have seen in the last section that there is a connection between supersingular elliptic curves over $\F_p$ with $\F_p$-rational endomorphism ring $\O$ and the ideal class group $\Cl(\O)$. We have used this information to discover an elegant structure for the isogeny graph of supersingular elliptic curves over $\F_p$.
Now we want to use this information to solve the isogeny problem.

Hence, let $E_0$ and $E_1$ be supersingular elliptic curves over $\F_p$, where $p > 3$.
As we have shown, every such elliptic curve corresponds to an ideal class in $\Cl(\O)$. Furthermore any rational $\ell$-isogeny from such an elliptic curve relates to an ideal of norm $\ell$ in $\Cl(\O)$. More precisely, if $E$ is represented by the ideal $\mathfrak a$ and the isogeny $\phi$ by the ideal $\mathfrak b$ with $\Norm(\mathfrak b) = \ell$, then the image curve $E'$ corresponds to the ideal $\mathfrak b^{-1}\cdot\mathfrak a$.

Due to a result from Bach~\cite{bach1985analytic}, using GRH, the ideals of norm less than or equal to $6\log(|d|)^2$ (where $d$ is the discriminant of $\O$) generate the ideal class group $\Cl(\O)$. Therefore we know that the supersingular isogeny graph is connected when we use all isogenies of prime degree up to $6\log(|d|)^2$. Usually we do not need all of those degrees and take a moderately small bound $B\leq 6 \log|d|^2$. One can determine in subexponential time whether a set of ideals generates the ideal class group. We set
\begin{eqnarray*}
  L & := & \left\{\text{primes }\ell < B \mid \left(\tfrac{-p}{\ell}\right) = 1\right\}.
\end{eqnarray*}

The condition $\left(\tfrac{-p}{\ell}\right) = 1$ comes from the fact that only in these cases there exist $\F_p$-rational $\ell$-isogenies between supersingular elliptic curves over $\F_p$, as shown in Theorem~\ref{thm_struct}.

Since we know about the volcano-like structure now, it is possible to adapt the usual ordinary-case-algorithm~\cite{galbraith2002extending} to this setting.
First we identify the endomorphism ring of the initial curves using Theorem~\ref{thm_struct} as in Kohel's algorithm.  If necessary we then take 2-isogenies so that both curves lie on the surface.
From now on we assume we have two supersingular elliptic curves over $\F_p$ with the same $\F_p$-rational endomorphism ring $\O_K$. From both vertices we perform a breadth-first search (or a random walk in a lower storage version) in the graph $X(\F_p, L)$ whose edges are isogenies of degree $\ell\leq B $. Since the graph is connected for a big enough bound $B\leq 6 \log(|d|)^2$, the algorithm invariably finds a path between the two vertices representing the elliptic curves. It is not hard to compute the whole isogeny as composition of small degree isogenies after that.

A very basic version of the high-storage bi-directional-search algorithm for computing a path in the subgraph of $X(\bar\F_p, L)$ is given as Algorithm~\ref{alg_path}.
By the  birthday paradox, the heuristic running time of the algorithm is $\tilde{O}( p^{1/4} )$ binary operations.

\begin{algorithm}[H]
  \caption{\ }
  \label{alg_path}
  \begin{algorithmic}[1]
    \REQUIRE Supersingular elliptic curves $E_0$, $E_1$ over $\F_p$, some bound $B \leq 6\log(|d|)^2$
    \STATE $S \gets [\ ]$
    \STATE Take vertical $2$-isogenies (if required) so that $E_0$ and $E_1$ are on the surface.
    \STATE $L \gets \left\{\text{primes }\ell < B \mid \left(\tfrac{-p}{\ell}\right) = 1\right\}$
    \STATE $S_0 \gets [j(E_0)]$, $S_1 \gets [j(E_1)]$
    \STATE $disjoint \gets \texttt{true}$
    \STATE $i \gets 0$
    \WHILE{$disjoint$}
      \STATE $\ell \stackrel{R}{\gets} L$
      \STATE $\Psi \gets \texttt{ModularPolynomial}(\ell)$
      \STATE $j \stackrel{R}{\gets} \texttt{Roots}(\Psi(X, j(E_i)))$
      \STATE $\texttt{Append}(S_i, j)$
      \IF{$j \in S_{1-i}$}
        \STATE $disjoint \gets \texttt{false}$
        \STATE $S \gets \texttt{Cat}(S_0[1, \dots, \texttt{Index}(S_0, j)], S_{1}[\texttt{Index}(S_1, j), \dots,  1])$
      \ENDIF
      \STATE $i \gets 1-i$
    \ENDWHILE
    \ENSURE A path $S$ in $X(\F_p, L)$ from $j(E_1)$ to $j(E_2)$
  \end{algorithmic}
\end{algorithm}

\begin{remark}
  $ \ \ $

  \noindent 1. Recall that the  supersingular isogeny graph $X( \F_p, L )$ has the property that each curve and its non-trivial quadratic twist give two distinct vertices.
  Hence, the graph is in some sense twice as large as we would like.
  Hence, in practice it is more convenient to forget about the non-isomorphic twists and just work with the $j$-invariants. This halves the number of vertices and furthermore we can use precomputed modular polynomials instead of computing the division polynomial of each elliptic curve in the chain. The resulting isogeny from $E_0$ can map to a quadratic twist of $E_1$, in which case we simply compose with a suitable isomorphism.

  \noindent 2. When $p \equiv 7 \pmod{8}$ one can use the prime $2$ in $L$, though one must be careful to identify which on of the three outgoing $2$-isogenies is actually going down to the floor.

  \noindent 3. There are many possible points of improvement like preferring small primes $\ell$ and using them more often~\cite{galbraith2011improved}, but to keep it simple they are omitted in this pseudo code.

  \noindent 4. A better algorithm would use Pollard-style random walks (i.e., walks that are deterministic and memoryless, so that when two walks collide they follow the same path from that point onwards) and distinguished points.  The details of such algorithms are given in~\cite{galbraith2002extending,galbraith2011improved}.
\end{remark}

We implemented \textsc{Algorithm}~\ref{alg_path} in \texttt{MAGMA}, as well as the standard high-storage bi-directional-search algorithm using the full graph $X( \Fpb, 2)$ for comparison. Table~\ref{table_results} shows the results of those computations. For each bit length we took ten random primes $p$ and for each prime selected $50$ random pairs of $j$-invariants in $\F_p$. The average lengths of a path in $X(\Fpb, 2)$ resp. $X(\F_p, L)$ for $L = \left\{\text{primes }\ell < 20 \mid \left(\tfrac{-p}{\ell}\right) = 1\right\}$ between the same pairs and the corresponding average CPU time in seconds are displayed.
The improvement from our new ideas is clear.

\begin{table}[H]

  \caption{Comparison of the average path length and running time for the bi-directional search algorithms in the full graph $X(\Fpb, 2)$ and the graph $X(\F_p, L)$ for random pairs of $j$-invariants in $\F_p$.}
  \begin{tabular}{c | c | c | c | c}
  \label{table_results}
      & \multicolumn{2}{c |}{path length} & \multicolumn{2}{c}{CPU time (seconds)} \\
    $p$ & $X(\Fpb, 2)$ & $X(\F_p, L)$ & $X(\Fpb, 2)$ & $X(\F_p, L)$ \\
    \hline \hline
    $16$-bit & 178 & 12 & 0.084 & 0.018 \\
    $20$-bit & 801 & 31 & 0.380 & 0.029 \\
    $24$-bit & 3234 & 51 & 2.021 & 0.083 \\
    $28$-bit & 13040 & 129 & 18.516 & 0.303 \\
    $32$-bit & 53118 & 235 & 325.852 & 0.720
  \end{tabular}
\end{table}

\section{The General Isogeny Problem}
\label{section_on}

We now consider the general isogeny problem: Given two supersingular elliptic curves $E_0$ and $E_1$ over $\Fpb$, to construct an isogeny between them.
We desire an algorithm that is easily distributed, that requires low storage, and that has total running time of $\tilde{O}( p^{1/2} )$ bit operations.

Such an algorithm can be developed using Pollard-style pseudorandom walks in the full graph, but the experience of the second author is that it is rather troublesome to implement, and the implied constants in the $\tilde{O}$ are poor.
Instead, we now have a much simpler approach: Run random walks in the graph from $E_0$ and $E_1$ until we hit a supersingular curve defined over $\F_p$. This step should require $\tilde{O}(p^{1/2})$ steps.  Then apply the new isogeny algorithm for supersingular elliptic curves over $\F_p$, which only requires $\tilde{O}(p^{1/4})$ steps.
The crucial point is that the first stage can be done with simple self-avoiding random walks -- rather than the much more difficult stateless Pollard-style walks.

In more detail, given $P$ processors one runs $P/2$ processors starting from each of $E_0$ and $E_1$ performing truly random self-avoiding walks (meaning that one remembers the current $j$-invariant $j_c$ and the previous $j$-invariant $j_p$, and at each step one chooses uniformly at random one of the roots $\Phi_\ell( j_c, Y)/(Y - j_p)$).
One could even instruct each of the processors to take a distinct path for the first $k = O( \log_2( P/2 ))$ steps (essentially computing distinct hash values on $k$-bit strings using the hash function of~\cite{charles2009cryptographic}).
Since the graph is an expander, we expect the walks to quickly be sampling uniformly from the graph, and so we expect to select a vertex in the subset of $j \in \F_p$ with probability approximately $p^{1/2} /p =  1/p^{1/2} $.
Alternatively, since the diameter is small, there should be a short path to the subset of $j \in \F_p$ of length $\tfrac{1}{2} \log(p )$ so one could distribute a depth-first search through all short paths from $E_0$.
In any case, we expect the first phase to be easily distributed, require little storage, and require total effort $\tilde{O}( p^{1/2} )$ bit operations (or total elapsed time $\tilde{O}( p^{1/2} /P)$ if we have $P$ processors of equal power).
The second stage has no effect on the asymptotic running time.

It is clear that this algorithm is much simpler to implement, and will have superior performance, to the approach using Pollard-style random walks in the full graph.
However, there is one disadvantage: The large storage or Pollard-style algorithms can work in the graph $X( \Fpb,  2  )$ and will find a sequence of $2$-isogenies from $E_0$ to $E_1$.  On the other hand, our algorithm in Section~\ref{section_iso_prob} for the subproblem where $j \in \F_p$ typically requires more primes, so the resulting isogeny is not a sequence of $2$-isogenies.  It is an open problem to transform an isogeny into a sequence of $2$-isogenies in the supersingular case (in the ordinary case this problem is the subgroup membership problem and discrete logarithm problem in the ideal class group).

\section*{Acknowledgements}

We thank David Kohel and Drew Sutherland for helpful conversations and Marco Streng for the idea of the proof of Proposition~\ref{prop_rational}.

Working on this paper started during a visit of the first author at the University of Auckland which was partially funded by a DAAD scholarship for PhD students.

\bibliographystyle{alpha} 
\bibliography{references}

\appendix
\section{Example Graphs}

We present a few small examples of the irregular structure of the full supersingular isogeny graph $X(\bar\F_p, \ell)$. After that we display, for the same examples, the graphs $X(\F_{p}, \ell)$ which have a much more regular structure. For the examples we use the primes $p = 83, 101$ and $103$, one for each of the different cases that occur. To demonstrate the two occurring structures we build the graphs for isogeny degrees $\ell=2$ and the smallest prime $\ell>2$ in each case for that isogenies exist.

Note that for $j(E) = 0$ resp. $j(E) \equiv 1728 \pmod p$ there are three resp. two non-equivalent isogenies mapping from $E$ to another curve $E'$. but their dual isogenies are all equivalent. This is due to the fact that $\#\Aut(E) =6$ resp. $\#\Aut(E) = 4$ in these cases. If $\phi:E\to E'$ is an isogeny and $\rho\in\Aut(E)$, then $\phi \circ \rho$ may not be equivalent (i.e., have the same kernel) as $\psi$, whereas the dual of $\phi\circ\rho$ is $\hat{\rho}\circ\hat{\phi}$, so this is equivalent to the dual of $\phi$. We denote these multiple isogenies in the graph using a single arrow together with an integer to indicate the multiplicity.

\subsection{An Example for $p\equiv 1 \pmod 4$}
\label{app_1mod4}

If we take $p=101$, we expect $\lfloor \tfrac{101}{12} \rfloor +1 = 9$ supersingular $j$-invariants in $\F_{p^2}$. In the next figure we show how they are connected using $2$-isogenies. The nodes labeled $\alpha$ and $\bar\alpha$ represent $j$-invariants in $\F_{p^2}\setminus\F_p$ where $\bar\alpha$ is the conjugate of $\alpha$. The graph can be easily computed with help of modular polynomials.

\begin{figure}[H]
  \caption{Supersingular Isogeny Graph $X(\bar \F_{101}, 2)$}
  \vspace{-4ex}
  \begin{displaymath}
    \xymatrix {
      & &
      \alpha \ar@<+.7ex>[rd] \ar@<+.7ex>[ld] \ar@<+.7ex>[d] & & & &
      \\
      0 \ar@<+.7ex>[r]|3 &
      66 \ar@<+.7ex>[l] \ar@<+.7ex>[rd] \ar@<+.7ex>[ru] &
      21 \ar@<+.7ex>[d] \ar@<+.7ex>[u] \ar@(dr, ur) &
      57 \ar@<+.7ex>[r] \ar@<+.7ex>[ld] \ar@<+.7ex>[lu] &
      64 \ar@<+.7ex>[l] \ar@<+.7ex>[r] \ar@/^1pc/[r] &
      3 \ar@<+.7ex>[l] \ar@<+.7ex>[r] \ar@/^1pc/[l]&
      59 \ar@<+.7ex>[l] \ar@(dl, dr) \ar@(ur, ul)
      \\
      & &
      \bar\alpha \ar@<+.7ex>[ru] \ar@<+.7ex>[lu] \ar@<+.7ex>[u] & & & &
    }
  \end{displaymath}
  \label{full101-2}
\end{figure}
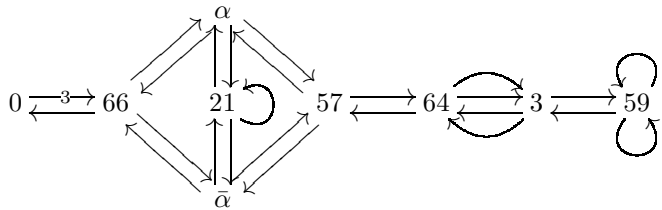

In $X(\F_p, \ell)$ we will have $h(-4p) = 14$ nodes which are supersingular elliptic curves over $\F_p$ with endomorphism ring $\Z[\sqrt{-101}]$. There will be only one outgoing $2$-isogeny from each curve, so naturally the graph can not be connected. It can be seen in the following figure.

\begin{figure}[H]
  \caption{$\F_p$-Rational Supersingular Isogeny Graph $X(\F_{101}, 2)$}
  \vspace{-4ex}
  \begin{displaymath}
    \xymatrix @R=.5pc {
      {\color{blue}   0 } \ar@<.7ex>[r] &
      {\color{blue}   66} \ar@<.7ex>[l] & & &
      {\color{blue}   57} \ar@<.7ex>[r] &
      {\color{blue}   64} \ar@<.7ex>[l] &
      {\color{blue}   3 } \ar@<.7ex>[r] &
      {\color{blue}   59} \ar@<.7ex>[l]
      \\
      & &
      {\color{black}  21} \ar@<.7ex>[r] &
      {\color{black}  21} \ar@<.7ex>[l]
      \\
      {\color{purple} 0 } \ar@<.7ex>[r] &
      {\color{purple} 66} \ar@<.7ex>[l] & & &
      {\color{purple} 57} \ar@<.7ex>[r] &
      {\color{purple} 64} \ar@<.7ex>[l] &
      {\color{purple} 3 } \ar@<.7ex>[r] &
      {\color{purple} 59} \ar@<.7ex>[l]
    }
  \end{displaymath}
  \label{rational101-2}
\end{figure}
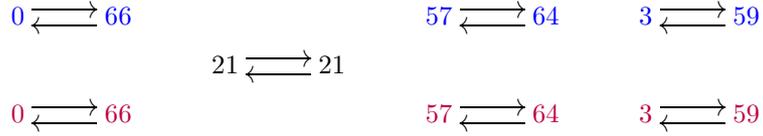

It is notable that in this graph there are fewer connecting isogenies than in the full graph before. For example, in the first graph we have two isogenies going from the node $64$ to the node $3$ and two ones back, which are all missing in the new graph. This is due to the fact that those isogenies are not defined over $\F_p$, so they are not computed as edges in $X(\F_p, 2)$. Likewise the two loops from $59$ to itself are isogenies over $\F_{p^2}$ that are dual to each other, whereas the loop at $21$ is a $\F_p$-rational isogeny that is its own dual.

For higher isogeny degrees the number of outgoing isogenies from each vertex grows, so the graph becomes more complicated to draw. For this example we can take $\ell=3$ since $\left(\tfrac{-p}{3}\right) = 1$.

\begin{figure}[H]
  \caption{Supersingular Isogeny Graph $X(\bar \F_{101}, 3)$\label{full101-3}}
  \vspace{-2ex}
  \begin{displaymath}
    \xymatrix {
      &
      \bar\alpha \ar@<+.7ex>[rr] \ar@<+.7ex>[dd] \ar@<+.7ex>[dr] \ar@/^1pc/[dr] & &
      59 \ar@<+.7ex>[ll] \ar@<+.7ex>[r] \ar@<+.7ex>[d] \ar@<+.7ex>[dl] &
      57 \ar@<+.7ex>[l] \ar@<+.7ex>[r] \ar@(ur, ul) \ar@(dl, dr) &
      66 \ar@<+.7ex>[l] \ar@<+.7ex>[ddll] \ar@<+.7ex>@/^1pc/[ddll] \ar@(dr, ur)
      \\
      & &
      \alpha \ar@<+.7ex>[ul] \ar@/^1pc/[ul] \ar@<+.7ex>[dl] \ar@<+.7ex>[ur]
      &
      21 \ar@<+.7ex>[u] \ar@<+.7ex>[d] \ar@(ul, dl) \ar@(dr, ur)
      \\
      0  \ar@<+.7ex>[r]|3 \ar@(ul, dl)  &
      64 \ar@<+.7ex>[l] \ar@<+.7ex>[uu] \ar@<+.7ex>[ur] \ar@<+.7ex>[rr] & &
      3 \ar@<+.7ex>[u] \ar@<+.7ex>[ll] \ar@<+.7ex>[uurr] \ar@<+.7ex>@/_1pc/[uurr]
      \\
    }
  \end{displaymath}
\end{figure}
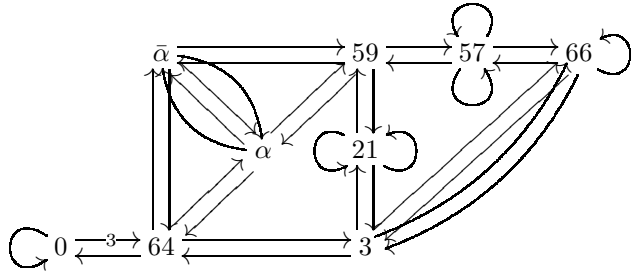

Despite the complicated picture of the full graph, the graph over $\F_p$ becomes just a big circle. In particular, it is already fully connected. This is because the ideal class group of $\Q( \sqrt{-101} )$ is generated by a prime ideal of norm $3$.

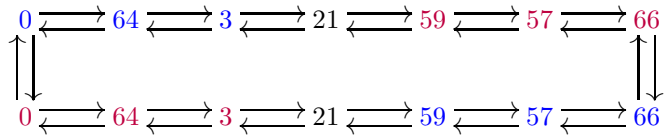
\begin{figure}[H]
  \caption{Rational Supersingular Isogeny Graph $X(\F_{101}, 3)$\label{rational101-3}}
  \vspace{-4ex}
  \begin{displaymath}
    \xymatrix {
      {\color{blue}   0 } \ar@<+.7ex>[r] \ar@<+.7ex>[d] &
      {\color{blue}   64} \ar@<+.7ex>[r] \ar@<+.7ex>[l] &
      {\color{blue}   3 } \ar@<+.7ex>[r] \ar@<+.7ex>[l] &
      {\color{black}  21} \ar@<+.7ex>[r] \ar@<+.7ex>[l] &
      {\color{purple} 59} \ar@<+.7ex>[r] \ar@<+.7ex>[l] &
      {\color{purple} 57} \ar@<+.7ex>[r] \ar@<+.7ex>[l] &
      {\color{purple} 66} \ar@<+.7ex>[d] \ar@<+.7ex>[l]
      \\
      {\color{purple} 0 } \ar@<+.7ex>[r] \ar@<+.7ex>[u] &
      {\color{purple} 64} \ar@<+.7ex>[r] \ar@<+.7ex>[l] &
      {\color{purple} 3 } \ar@<+.7ex>[r] \ar@<+.7ex>[l] &
      {\color{black}  21} \ar@<+.7ex>[r] \ar@<+.7ex>[l] &
      {\color{blue}   59} \ar@<+.7ex>[r] \ar@<+.7ex>[l] &
      {\color{blue}   57} \ar@<+.7ex>[r] \ar@<+.7ex>[l] &
      {\color{blue}   66} \ar@<+.7ex>[u] \ar@<+.7ex>[l]
    }
  \end{displaymath}
\end{figure}

Again you can see how the isogenies from the full graph that are defined over $\F_{p^2}$ vanish in the rational graph, and the single loops become isogenies from an elliptic curve to its quadratic twist. This latter fact can be shown in general.

\subsection{An Example for $p\equiv 3 \pmod 8$}
\label{app_3mod8}

For this case we take $p=83$, so the full graph will have $\lfloor\frac{83}{12}\rfloor + 2 = 8$ vertices. Again we have two $j$-invariants $\alpha, \bar\alpha \in \F_{p^2}\setminus\F_p$. The full $2$-isogeny graph has the following structure.

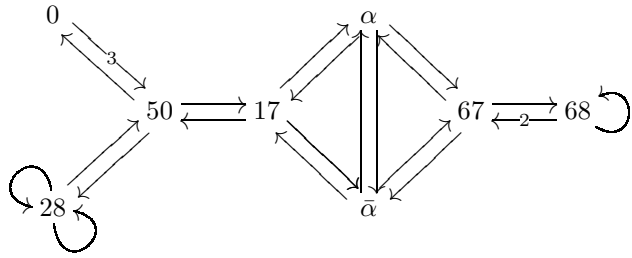
\begin{figure}[H]
  \caption{Supersingular Isogeny Graph $X(\bar \F_{83}, 2)$\label{full83_2}}
  \vspace{-4ex}
  \begin{displaymath}
    \xymatrix {
      0 \ar@<+.7ex>[rd]|
      3 & & &
      \alpha \ar@<+.7ex>[rd] \ar@<+.7ex>[ld] \ar@<+.7ex>[dd]& &
      \\
      & 50 \ar@<+.7ex>[ld] \ar@<+.7ex>[lu] \ar@<+.7ex>[r] &
      17 \ar@<+.7ex>[l] \ar@<+.7ex>[rd]\ar@<+.7ex>[rd] \ar@<+.7ex>[ru]& &
      67 \ar@<+.7ex>[ld] \ar@<+.7ex>[lu] \ar@<+.7ex>[r]&
      68 \ar@<+.7ex>[l]|2 \ar@(dr, ur)
      \\
      28 \ar@<+.7ex>[ru] \ar@(u, l) \ar@(d, r) & & &
      \bar\alpha \ar@<+.7ex>[ru] \ar@<+.7ex>[lu] \ar@<+.7ex>[uu] & &
    }
  \end{displaymath}
\end{figure}

In the graph over $\F_p$ we get $h(-p) = 3$ supersingular elliptic curves on the surface and $h(-4p)=9$ ones on the floor. In the next figure we can see how $2$-isogenies connect floor and surface as explained in case (2)(b) of Theorem~\ref{thm_struct}.

\begin{figure}[H]
  \caption{Rational Supersingular Isogeny Graph $X(\F_{83}, 2)$\label{rational83_2}}
  \vspace{-4ex}
  \begin{displaymath}
    \xymatrix {
      &
      {\color{blue}   50} \ar@<.7ex>[d] \ar@<.7ex>[dr] \ar@<.7ex>[dl] & & &
      {\color{purple} 50} \ar@<.7ex>[d] \ar@<.7ex>[dr] \ar@<.7ex>[dl] & & &
      {\color{black}  68} \ar@<.7ex>[d] \ar@<.7ex>[dr] \ar@<.7ex>[dl]
      \\
      {\color{blue}   0} \ar@<.7ex>[ur]&
      {\color{blue}   17} \ar@<.7ex>[u]&
      {\color{blue}   28} \ar@<.7ex>[ul]&
      {\color{purple} 0} \ar@<.7ex>[ur]&
      {\color{purple} 17} \ar@<.7ex>[u]&
      {\color{purple} 28} \ar@<.7ex>[ul]&
      {\color{black}  67} \ar@<.7ex>[ur]&
      {\color{black}  67} \ar@<.7ex>[u]&
      {\color{black}  68} \ar@<.7ex>[ul]
    }
  \end{displaymath}
\end{figure}
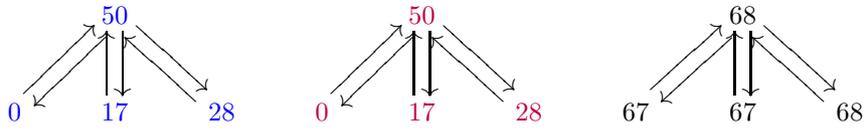

If we repeat the procedure for $\ell=3$, the full graph looks like this.

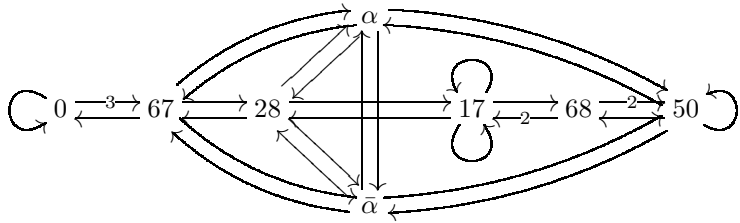
\begin{figure}[H]
  \caption{Supersingular Isogeny Graph $X(\bar \F_{83}, 3)$\label{full83_3}}
  \vspace{-4ex}
  \begin{displaymath}
    \xymatrix {
      & & &
      \alpha \ar@<+.7ex>[dd] \ar@<+.7ex>@/_.7pc/[dll] \ar@<+.7ex>[dl] \ar@<+.7ex>@/^.7pc/[drrr]
      \\
      0  \ar@(ul, dl) \ar@<+.7ex>[r]|3 &
      67 \ar@<+.7ex>[l] \ar@<+.7ex>[r] \ar@<+.7ex>@/_.7pc/[drr] \ar@<+.7ex>@/^.7pc/[urr] &
      28 \ar@<+.7ex>[l] \ar@<+.7ex>[rr] \ar@<+.7ex>[dr] \ar@<+.7ex>[ur] & &
      17 \ar@<+.7ex>[ll] \ar@<+.7ex>[r] \ar@(ur, ul) \ar@(dl, dr) &
      68 \ar@<+.7ex>[l]|2 \ar@<+.7ex>[r]|2 &
      50 \ar@<+.7ex>[l] \ar@(dr, ur) \ar@<+.7ex>@/^.7pc/[dlll] \ar@<+.7ex>@/_.7pc/[ulll]
      \\
      & & &
      \bar\alpha \ar@<+.7ex>[uu] \ar@<+.7ex>[ul] \ar@<+.7ex>@/^.7pc/[ull] \ar@<+.7ex>@/_.7pc/[urrr]
    }
  \end{displaymath}
\end{figure}

And in the graph over $\F_p$ we get two isogeny circles, one on the floor and one on the surface.

\begin{figure}[H]
  \caption{Rational Supersingular Isogeny Graph $X(\F_{83}, 3)$\label{rational83_3}}
  \vspace{-4ex}
  \begin{displaymath}
    \xymatrix @R=.5pc {
      & & &
      {\color{black}  68} \ar@<+.7ex>[ddr] \ar@<+.7ex>[ddl]
      \\
      \\
      & &
      {\color{blue}   50} \ar@<+.7ex>[rr] \ar@<+.7ex>[uur] & &
      {\color{purple} 50} \ar@<+.7ex>[uul] \ar@<+.7ex>[ll]
      \\
      \\
      &
      {\color{blue}   17} \ar@<+.7ex>[dl] \ar@<+.7ex>[rr] & &
      {\color{black}  68} \ar@<+.7ex>[rr] \ar@<+.7ex>[ll] & &
      {\color{purple} 17} \ar@<+.7ex>[ll] \ar@<+.7ex>[dr]
      \\
      {\color{purple} 28} \ar@<+.7ex>[dd] \ar@<+.7ex>[ur] & & & & & &
      {\color{blue}   28} \ar@<+.7ex>[dd] \ar@<+.7ex>[ul]
      \\
      \\
      {\color{black}  67} \ar@<+.7ex>[uu] \ar@<+.7ex>[drr] & & & & & &
      {\color{black}  67} \ar@<+.7ex>[dll] \ar@<+.7ex>[uu]
      \\
      & &
      {\color{blue}   0} \ar@<+.7ex>[rr] \ar@<+.7ex>[ull] & &
      {\color{purple} 0} \ar@<+.7ex>[ll] \ar@<+.7ex>[urr]
    }
  \end{displaymath}
\end{figure}
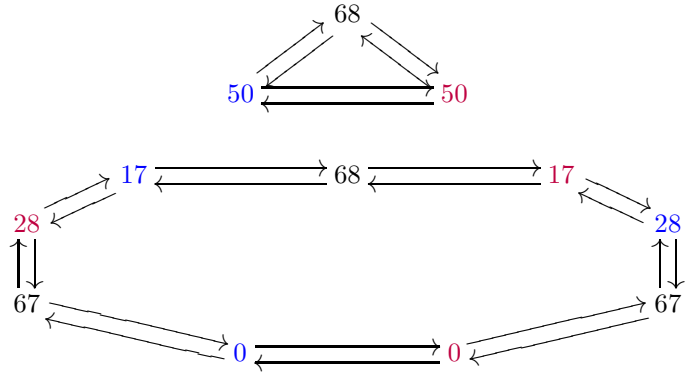

\subsection{An Example for $p\equiv 7 \pmod 8$}
\label{app_7mod8}

Our example here is $p=103$ where we have $h(-p) = 5$ supersingular elliptic curves on the surface and also $h(-4p) = 5$ ones on the floor. In this case we have four nodes in $\F_{p^2}\setminus\F_p$.

\begin{figure}[H]
  \caption{Supersingular Isogeny Graph $X(\bar \F_{103}, 2)$\label{full103-2}}
  \vspace{-2ex}
  \begin{displaymath}
    \xymatrix {
      24 \ar@(r, u) \ar@(l, d) \ar@<.7ex>[rd] & & & &
      \alpha \ar@<.7ex>[ld] \ar@<.7ex>[dd] \ar@<.7ex>[r] &
      \beta \ar@<.7ex>[l] \ar@<.7ex>[dd] \ar@/^1.5pc/[dd]
      \\
      &
      23 \ar@<.7ex>[ld] \ar@<.7ex>[lu] \ar@<.7ex>[r] &
      69 \ar@<.7ex>[l] \ar@<.7ex>[r] \ar@(ld, rd) &
      34 \ar@<.7ex>[l] \ar@<.7ex>[rd] \ar@<.7ex>[ru] & &
      \\
      80 \ar@<.7ex>[ru]|2 \ar@(d, l) & & & &
      \bar\alpha \ar@<.7ex>[lu] \ar@<.7ex>[uu] \ar@<.7ex>[r] &
      \bar\beta \ar@<.7ex>[l] \ar@<.7ex>[uu] \ar@/^1.5pc/[uu]
    }
  \end{displaymath}
\end{figure}
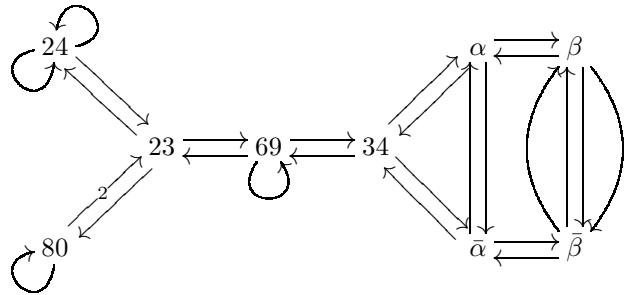

The $2$-isogeny graph over $\F_p$ in this case is already connected. Again a volcano structure can be observed where every supersingular elliptic curve on the floor has exactly one isogeny up starting at it.

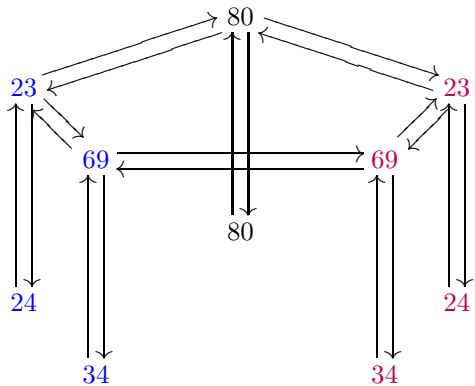
\begin{figure}[H]
  \caption{Rational Supersingular Isogeny Graph $X(\F_{103}, 2)$\label{rational103-2}}
  \vspace{-4ex}
  \begin{displaymath}
    \xymatrix @!=.25pc {
      & & &
      {\color{black}  80} \ar@<.7ex>[ddd] \ar@<.7ex>[drrr] \ar@<.7ex>[dlll]
      \\
      {\color{blue}   23} \ar@<.7ex>[ddd] \ar@<.7ex>[urrr] \ar@<.7ex>[dr]& & & & & &
      {\color{purple} 23} \ar@<.7ex>[ddd] \ar@<.7ex>[ulll] \ar@<.7ex>[dl]
      \\
      &
      {\color{blue}   69} \ar@<.7ex>[ddd] \ar@<.7ex>[rrrr] \ar@<.7ex>[ul]& & & &
      {\color{purple} 69} \ar@<.7ex>[ddd] \ar@<.7ex>[ur] \ar@<.7ex>[llll]
      \\
      & & &
      {\color{black}  80} \ar@<.7ex>[uuu]
      \\
      {\color{blue}   24} \ar@<.7ex>[uuu] & & & & & &
      {\color{purple} 24} \ar@<.7ex>[uuu]
      \\
      &
      {\color{blue}   34} \ar@<.7ex>[uuu] & & & &
      {\color{purple} 34} \ar@<.7ex>[uuu]
    }
  \end{displaymath}
\end{figure}

The smallest prime $\ell>2$ with $\left(\tfrac{-103}{\ell}\right) = 1$ is $\ell=7$. In the full graph every vertex has eight outgoing isogenies so it is not nice to draw. The subgraph of $X(\bar \F_{103},7)$ only consisting of $j$-invariants in $\F_{103}$ is presented in the next figure, so it can be compared to $X(\F_{103},7)$ below.

\begin{figure}[H]
  \caption{Subgraph of Supersingular Isogeny Graph $X(\bar \F_{103}, 7)$\label{full103-7}}
  \vspace{-4ex}
  \begin{displaymath}
    \xymatrix {
      \\
      23 \ar@(ul, dl) \ar@<.7ex>[r] &
      69 \ar@<.7ex>[r] \ar@<.7ex>[l] \ar@<.7ex>@/^1.5pc/[rr] \ar@<.7ex>@/_1.5pc/[rr] &
      80 \ar@<.7ex>[r]|2 \ar@<.7ex>[l]|2 &
      34 \ar@<.7ex>[r] \ar@<.7ex>[l] \ar@(ur, ul) \ar@(dl, dr) \ar@<.7ex>@/^1.5pc/[ll] \ar@<.7ex>@/_1.5pc/[ll] &
      24 \ar@<.7ex>[l] \ar@(dr, ur)
    }
  \end{displaymath}
\end{figure}
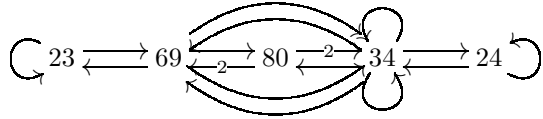

Again we get two isogeny cycles such that floor and surface each are fully connected when we draw the graph $X(\F_{103}, 7)$. This is because the ideal class group is cyclic and generated by a prime ideal of norm $7$.

\begin{figure}[H]
  \caption{Rational Supersingular Isogeny Graph $X(\F_{103}, 7)$\label{rational103-7}}
  \vspace{-4ex}
  \begin{displaymath}
    \xymatrix @!=.25pc {
      & & &
      {\color{black}  80} \ar@<.7ex>[drrr] \ar@<.7ex>[dlll]
      \\
      {\color{blue}   69} \ar@<.7ex>[urrr] \ar@<.7ex>[dr] & & & & & &
      {\color{purple} 69} \ar@<.7ex>[ulll] \ar@<.7ex>[dl]
      \\
      &
      {\color{purple} 23} \ar@<.7ex>[rrrr] \ar@<.7ex>[ul] & & & &
      {\color{blue}   23} \ar@<.7ex>[ur] \ar@<.7ex>[llll]
      \\
      & & &
      {\color{black}  80} \ar@<.7ex>[drrr] \ar@<.7ex>[dlll]
      \\
      {\color{blue}   34} \ar@<.7ex>[urrr] \ar@<.7ex>[dr] & & & & & &
      {\color{purple} 34} \ar@<.7ex>[ulll] \ar@<.7ex>[dl]
      \\
      &
      {\color{purple} 24} \ar@<.7ex>[rrrr] \ar@<.7ex>[ul] & & & &
      {\color{blue}   24} \ar@<.7ex>[ur] \ar@<.7ex>[llll]
    }
  \end{displaymath}
\end{figure}
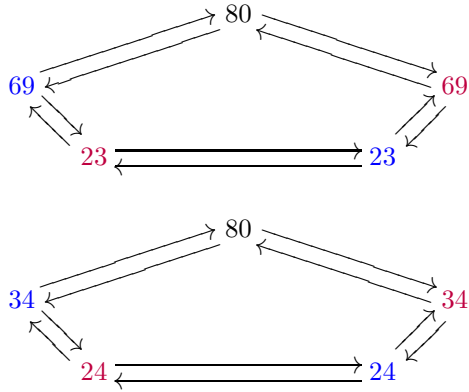

\vfill
\end{document}